\documentclass[12pt]{amsart}

\usepackage[margin=1.5in]{geometry}

\def\Q{{\mathbb{Q}}}

\def\C{{\mathbb{C}}}

\usepackage{amsmath}
\usepackage{amsfonts}
\usepackage{amssymb}
\usepackage{xy}
\usepackage{longtable}
\usepackage{amsthm}
\usepackage{enumerate}
\usepackage{graphicx}
\usepackage{mathrsfs}
\usepackage{mathrsfs}
\usepackage{xcolor}

\theoremstyle{plain}
\newtheorem{theorem}{Theorem}[section]
\newtheorem{corollary}[theorem]{Corollary}
\newtheorem{lemma}[theorem]{Lemma}
\newtheorem{proposition}[theorem]{Proposition}
\theoremstyle{definition}
\newtheorem{definition}[theorem]{Definition}
\newtheorem{example}[theorem]{Example}
\newtheorem{rem}[theorem]{Remark}

\makeatletter
\def\ps@pprintTitle{%
  \let\@oddhead\@empty
  \let\@evenhead\@empty
  \let\@oddfoot\@empty
  \let\@evenfoot\@oddfoot
}
\makeatother

\title{Counting polarizations on abelian varieties with group action}

\author{Robert Auffarth, Angel Carocca, Rub\'{\i} E. Rodr\'{\i}guez}
\address{R. Auffarth \\Departamento de Matem\'aticas, Facultad de
Ciencias, Universidad de Chile, Santiago\\Chile}
\email{rfauffar@uchile.cl}
\address{A. Carocca \\Departamento de Matemática y Estadística, Universidad de La Frontera. Temuco, Chile}
\email{angel.carocca@ufrontera.cl}
\address{R. E. Rodr\'{\i}guez\\ Departamento de Matemática y Estadística, Universidad de La Frontera. Temuco, Chile}
\email{rubi.rodriguez@ufrontera.cl}

\begin{document}

\begin{abstract}
Let $\mathcal{A}_g$ be the moduli space of principally polarized abelian varieties. We study the problem of counting the number of principal polarizations modulo the natural action of the automorphism group of the abelian variety on a very general element of a positive dimensional component of $\mathrm{Sing}(\mathcal{A}_g)$, and show that this number is not always 1.
\end{abstract}

\maketitle

\section{Introduction}

Given a complex principally polarized abelian variety $(A,\mathcal{L})$ of dimension $g$, it is natural to ask how many principal polarizations $A$ can have, modulo the action of the automorphism group of $A$ by pullback. This is equivalent to asking how many times the variety $A$ appears in the moduli space $\mathcal{A}_g$ of principally polarized abelian varieties of dimension $g$. We denote this number by $\pi(A)$. A classical result by Narasimhan and Nori states that $\pi(A)$ is finite, but it can be shown, for example, that it is neither a semi-continuous nor bounded function on $\mathcal{A}_g$ for $g\geq 2$ (see \cite[Theorem 5.1]{Lange2}, for instance). Another interesting property of $\pi(A)$ is that, as one might expect, it can be quite large on products of elliptic curves. For example, if $E$ is an elliptic curve without complex multiplication, then $\pi(E^g)=1$ if $g\leq 7$, but $\pi(E^{23})=117$ and $\pi(E^{32})\geq80000000$ (see \cite{Lange2})!

Many interesting examples have also been constructed in the context of the Torelli Theorem for smooth projective curves. Indeed, it is natural to ask if, given a smooth projective curve $C$, its (non-polarized) Jacobian variety $JC$ uniquely characterizes $C$. The answer is a resounding no, with interesting examples already appearing in genus 2 and 3 (see \cite{Howe} and \cite{Abdellatif}). Even in genus 3, Howe \cite{Howe} shows that the non-polarized Jacobian variety doesn't even determine if the curve is hyperelliptic or not. This shows that Jacobian varieties can have many non-isomorphic principal polarizations.

One of the first in-depth studies of the number $\pi(A)$ was done by Lange in \cite{Lange}; we will briefly recall his main result. If $A$ is an abelian variety, let $\mathrm{Aut}_+^s(A)$ denote the set of automorphisms of $A$ that are fixed by the Rosati involution $\dagger$ and whose minimal polynomial has strictly positive roots. Lange proved the following:

\begin{theorem}
If $A$ has a principal polarization, then
\[\pi(A)=\#(\mathrm{Aut}_+^s(A)/\mathrm{Aut}(A)),\] 
where $\mathrm{Aut}(A)$ acts on $\mathrm{Aut}_+^s(A)$ by $\mu:\sigma\mapsto \mu^\dagger\sigma \mu$.\footnote{Note that $\mathrm{Aut}_+^s(A)$ is not usually a group, and  $\mathrm{Aut}_+^s(A)/\mathrm{Aut}(A)$ denotes the quotient of  the \textit{set} $\mathrm{Aut}_+^s(A)$ by the action of $\mathrm{Aut}(A)$.}
\end{theorem}

This shows in general terms that there is an intimate relation between principal polarizations on an abelian variety and the existence of certain automorphisms on the abelian variety. In this line of reasoning, recall that for $g\geq 3$, $\mathrm{Sing}(\mathcal{A}_g)$ is the locus of principally polarized abelian varieties that have an automorphism different from $\pm1$ that preserves the polarization. It is then natural to ask how many principal polarizations an element of $\mathrm{Sing}(\mathcal{A}_g)$ can have. This question, in this generality, is too difficult to give any reasonable answer, and so in this article we concentrate on the following question:\\

\noindent\textbf{Question:} Let $Z\subseteq\text{Sing}(\mathcal{A}_g)$ be a positive-dimensional component. What is $\pi(A)$ for a very general\footnote{Here we note that the term \textit{very general} is somewhat vague, and we will make this more precise in the sections that follow (see Corollary \ref{verygeneral} and its proof for precise details).} element $(A,\mathcal{L})\in Z$?\\

In order to attack this problem, it soon becomes clear that it is necessary to not only know how to count the number of principal polarizations on an abelian variety, but also the number of polarizations of any given (non-principal) type. To this end, our first result generalizes Lange's result on the number of principal polarizations on an abelian variety to the number of polarizations of arbitrary type. If $D=(d_1,\ldots,d_g)$ with $d_i\in\mathbb{Z}_{>0}$ and $d_i\mid d_{i+1}$ for $i\leq g-1$, then define
\[\pi_D(A):=\#\{\text{polarizations of type }D\text{ on }A\}/\mathrm{Aut}(A),\]
where $\mathrm{Aut}(A)$ acts by pullback. If $\alpha=[\mathcal{L}]$ for $\mathcal{L}\in\mathrm{Pic}(A)$ is a polarization, as usual we define 
\[K(\alpha):=\{x\in X:t_x^*\mathcal{L}\simeq\mathcal{L}\}.\]
Now if $\mathcal{H}_D$ denotes the set of all finite subgroups $K$ of $A$ such that there exists a polarization $\alpha$ of type $D$ with $K(\alpha)=K$, and
\[G_K:=\{\sigma\in\mathrm{Aut}(A):\sigma(K)=K\},\]
then we obtain our first main result:\\

\begin{theorem}
If $\mathrm{Aut}_+^s(A^\vee)$ denotes the set of all totally positive automorphisms on the dual abelian variety fixed by the Rosati involution given by the dual polarization, then
\[\pi_D(A)=\sum_{[K]\in\mathcal{H}_D/\mathrm{Aut}(A)}\#(\mathrm{Aut}_+^s(A^\vee)/G_K),\]
where $G_K$ acts on $\mathrm{Aut}_+^s(A^\vee)$ by  $\mu:\eta\mapsto\mu^\vee\eta(\mu^\dagger)^\vee$.
\end{theorem}

As an interesting and unexpected corollary, we obtain the following (see Corollary \ref{ppmin} for the proof):

\begin{corollary}
If $A$ has a polarization of type $D$, then $\pi(A)\leq\pi_D(A)$.
\end{corollary}

Now let's return to our main motivation. By the main result of \cite{GMZ}, there exists a prime number $p$, a natural number $n\geq 3$ such that $p \nmid n$ and an element $\rho_n\in\mathrm{GL}(2g,\mathbb{Z}/n\mathbb{Z})$ such that if $Z$ is an irreducible component of $\mathrm{Sing}(\mathcal{A}_g)$, then
\[Z=\mathcal{A}_g(p,\rho_n):=\{(A,\mathcal{L})\in\mathcal{A}_g:\exists\sigma\in\mathrm{Aut}(A,\mathcal{L}),|\sigma|=p,\sigma|_{A[n]}\sim\rho_n\},\]
where $\sim$ denotes equivalence of representations (i.e. in this case, conjugation). In particular, for every $(A,\mathcal{L})\in Z$, $\mathbb{Z}[\zeta_p]\subseteq\mathrm{End}(A)$ where $\zeta_p$ is a primitive $p$th root of unity. Let $U_p^+$ denote the subgroup of $U_p:=\mathbb{Z}[\zeta_p+\zeta_p^{-1}]^\times$ that consists of totally positive units; note that $U_p^2\subseteq U_p^+$. Define
\[\mathfrak{u}(p):=|U_p^+/U_p^2|,\] 
which is just the quotient between the narrow class number and the class number of the field $\mathbb{Q}(\zeta_p+\zeta_p^{-1})$.

Again following \cite{GMZ}, to $Z$ we can associate the list $n_0,n_1,\ldots,n_{p-1}$, where if $(A,\mathcal{L})\in Z$ is general and has an automorphism $\sigma$ of order $p$, then $n_i$ is the dimension of the eigenspace for $d\sigma$ associated to the eigenvalue $\zeta_p^i$ on $T_0(A)$. Using local deformation theory, it is shown that
\[\dim Z=\frac{n_0(n_0+1)}{2}+\sum_{i=1}^{(p-1)/2}n_in_{p-i}.\]
Our main theorem on the number of principal polarizations a general element of a component can have states the following:

\begin{theorem}\label{main}
Let $Z=\mathcal{A}_g(p,\rho_n)\subseteq\mathrm{Sing}(\mathcal{A}_g)$ be a positive-dimensional irreducible component of the singular locus of $\mathcal{A}_g$.  
\begin{enumerate}
\item If $p=2$, then $\pi(A)=1$ for a very general member of $Z$. 
\item If $p>2$, $\sum_{i=1}^{(p-1)/2}n_in_{p-i}>0$ and $n_i\neq 1$ for some $i$, then for a very general element $(A,\mathcal{L})\in Z$, $\pi(A)=\mathfrak{u}(p)$.
\end{enumerate}
Moreover, if $(A,\mathcal{L})\in Z$ is very general, then for every principal polarization $\mathcal{L}'$ on $A$, $(A,\mathcal{L}')\in Z$.
\end{theorem}

The numerical conditions imposed in the theorem are not particularly restrictive; however, we do not know what happens if these are not satisfied, since we lose control over the endomorphism algebra of a very general member of the component. 

Recently Dummit and Kisilevsky \cite[Theorem 7]{DK} proved that $\mathfrak{u}(n)$ can be arbitrarily large for (highly composite) $n\in\mathbb{N}$ (the definition of $\mathfrak{u}(n)$ being the same for non-primes), but it is an open question whether or not there are even infinitely many primes $p$ where $\mathfrak{u}(p)\neq1$. It can be checked that for all primes $p$ between $2$ and $500$, $\mathfrak{u}(p)=1$ except for the primes found in the following table:\\

\begin{center}
\begin{tabular}{ c | c| c| c| c| c| c| c| c| c| c| c| c }
$p$&29&113&197&239&277&311&337&373&397&421&463&491\\\hline
$\mathfrak{u}(p)$&8&8&8&8&4&1024&64&32&4&16&8&8
\end{tabular}
\end{center}

\vspace{0.5cm}

In the last section, we use this information, along with a result by Zarhin \cite{Zarhin}, to describe examples where the very general element of a component of $\mathrm{Sing}(\mathcal{A}_g)$ has many principal polarizations.\\
 
\noindent\textit{Acknowledgements:} We would like to thank Eduardo Friedman for many helpful discussions and David Dummit and Hershy Kisilevsky for their interest in this work. The first author was partially supported by the ANID - Fondecyt grant 1220997, the second author was partially supported by the ANID - Fondecyt grant 1240181, and the third author was partially supported by the ANID - Fondecyt grant 1230708. 

\section{Counting polarizations of a given type on an abelian variety}

Let $(A,\alpha)$ be a complex polarized abelian variety of type $D:=(d_1,\ldots,d_g)$. For our first order of business, following Lange's idea in \cite{Lange}, we would like to count the set
\[\mathcal{P}_D:=\{\beta\in\mathrm{NS}(A):\beta\text{ is a polarization of type }D\}/\mathrm{Aut}(A),\]
where $\mathrm{Aut}(A)$ acts by pullback of numerical classes. For every polarization $\beta\in\mathrm{NS}(A)$ of type $D$ induced by an ample line bundle $\mathcal{L}\in\mathrm{Pic}(A)$, we have an isogeny
\[\varphi_\beta:A\to A^\vee=\mathrm{Pic}^0(A)\]
\[x\mapsto t_x^*\mathcal{L}\otimes\mathcal{L}^{-1}\]
whose kernel is 
\[K(\beta):=\ker(\varphi_\beta)\simeq\left(\mathbb{Z}/d_1\mathbb{Z}\times\cdots\times\mathbb{Z}/d_g\mathbb{Z}\right)^2.\] We will say that $K(\beta)$ is a \textit{subgroup of type} $D$. If $K\leq A$ is a subgroup of type $D$ of $A$, then define the following sets:
\[G_K:=\{\sigma\in\mathrm{Aut}(A):\sigma(K)=K\}\]
\[\mathcal{P}_D(K):=\{\beta\in\mathrm{NS}(A)\text{ ample of type }D, K(\beta)=K\}/G_K.\] 

\begin{definition} Define the numbers
\[\pi_D(A):=\#\mathcal{P}_D\]
\[\pi(A)=\pi_I(A):=\#\text{ of principal polarizations on }A.\]
\end{definition}

We can relate $\mathcal{P}_D$ and $\mathcal{P}_D(K)$ as follows:

\begin{proposition}
We have that $\mathcal{P}_D(K)$ can be identified with the set of all elements $[\beta]\in\mathcal{P}_D$ such that there exists an automorphism $\sigma\in\mathrm{Aut}(A)$ such that $K(\sigma^*\beta)=\sigma^{-1}(K(\beta))=K$. 
\end{proposition}
\begin{proof}
There is clearly a map from $\mathcal{P}_D(K)$ to $\mathcal{P}_D$ whose image consists of all $[\beta]$ such that there exists an automorphism $\sigma\in\mathrm{Aut}(A)$ such that $K(\sigma^*\beta)=K$. Now if $\beta,\gamma\in\mathrm{NS}(A)$ are of type $D$ and $K(\beta)=K(\gamma)=K$, then their images in $\mathcal{P}_D$ are equal if and only if there exists $\sigma\in\mathrm{Aut}(A)$ such that $\beta=\sigma^*\gamma$. Moreover, 
\[K=\ker(\varphi_\beta)=\ker(\varphi_{\sigma^*\gamma})=\ker(\sigma^\vee\varphi_\gamma\sigma)=\sigma^{-1}(K(\gamma))=\sigma^{-1}(K).\]
Therefore $\sigma\in G_K$ and we conclude that the map $\mathcal{P}_D(K)\to\mathcal{P}_D$ is injective.
\end{proof}

We will say that an endomorphism $f\in\mathrm{End}(A)$ is \textit{totally positive} if its minimal polynomial has strictly positive roots. This notion can also clearly be extended to $\mathbb{Q}$-endomorphisms. We will now prove the following theorem, which is a direct generalization of Lange's main result in \cite{Lange}:

\begin{theorem}\label{counttheorem}
If $K\leq A$ is a subgroup of type $D$ and $\alpha_0$ is a polarization with $K(\alpha_0)=K$, there is a bijection
\[\mathcal{P}_D(K)\longleftrightarrow\mathrm{Aut}_+^s(A^\vee)/G_K,\]
where $\mathrm{Aut}_+^s(A^\vee)$ consists of all totally positive automorphisms of $A$ that are symmetric with respect to the Rosati involution of the dual polarization $\alpha_0^\vee$ on $A^\vee$, and $G_K$ acts on $\mathrm{Aut}_+^s(A^\vee)$ by $\mu:\eta\mapsto\mu^\vee\eta(\mu^\dagger)^\vee$. 
\end{theorem}
\begin{proof}

By \cite[Remark 5.2.5]{BL}, the map $\beta\mapsto\varphi_{\alpha_0}^{-1}\varphi_\beta$ gives a bijection between the set of polarizations on $A$ and the set of totally positive symmetric elements of $\mathrm{End}_\mathbb{Q}(A)$. Assume now that $K(\beta)=K$. By conjugating this element by $\varphi_{\alpha_0}$, we get the element $\sigma_\beta:=\varphi_\beta\varphi_{\alpha_0}^{-1}$. Note that since $\varphi_{\alpha_0}$ and $\varphi_\beta$ have the same kernel, if $\Lambda^\vee$ denotes the lattice of $A^\vee$ in $T_0(A^\vee)$, we have that
\[\rho_a(\varphi_\beta)(\rho_a(\varphi_{\alpha_0}^{-1}(\Lambda^\vee)))\subseteq\Lambda^\vee.\]
In particular, $\sigma_\beta$ is an honest endomorphism of $A^\vee$, and by the same reasoning so is its inverse. Therefore, $\sigma_\beta$ is an automorphism of $A^\vee$. 

Now, since the dual polarization satisfies $\varphi_{\alpha_0^\vee}=d\varphi_{\alpha_0}^{-1}$ for $d$ the exponent of $\varphi_{\alpha_0}$, it is a trivial calculation to check that $\sigma_{\beta}$ is symmetric with respect to $\alpha_0^\vee$:
\[\sigma_{\beta}^\dagger=\varphi_{\alpha_0^\vee}^{-1}\sigma_{\beta}^\vee\varphi_{\alpha_0^\vee}=\frac{1}{d}\varphi_{\alpha_0}\varphi_{\alpha_0}^{-1}\varphi_\beta d\varphi_{\alpha_0}^{-1}=\sigma_\beta.\]
Moreover, since we conjugated $\varphi_{\alpha_0}^{-1}\varphi_\beta$ by an isogeny to obtain $\sigma_\beta$, we see that these two elements have the same characteristic polynomial, which implies that $\sigma_\beta$ is a totally positive automorphism. 

Now if $\sigma\in\mathrm{Aut}_+^s(A^\vee)$ is a totally positive symmetric automorphism of $A^\vee$, we get that $\sigma\varphi_{\alpha_0}$ is self dual, since
\[(\sigma\varphi_{\alpha_0})^\vee=\varphi_{\alpha_0}\sigma^\vee=\varphi_{\alpha_0}\varphi_{\alpha_0^\vee}\varphi_{\alpha_0^\vee}^{-1}\sigma^\vee\varphi_{\alpha_0^\vee}\varphi_{\alpha_0^\vee}^{-1}=\varphi_{\alpha_0}\varphi_{\alpha_0^\vee}\sigma^\dagger\varphi_{\alpha_0^\vee}^{-1}=\sigma\varphi_{\alpha_0}.\]
In particular, $\sigma\varphi_{\alpha_0}=\varphi_\beta$ for some $\beta\in\mathrm{NS}(A)$. Since the roots of the characteristic polynomial of $\sigma$ are positive, we have that $\beta$ is ample, and clearly of type $D$ since $\alpha_0$ is. If $\mu\in G_K$, then $\varphi_{\mu^*\beta}=\mu^\vee\varphi_{\beta}\mu$, and so the action of $G_K$ on the polarization translates into the following action of $G_K$ on $\mathrm{Aut}_+^s(A^\vee)$:
\[\mu:\sigma\mapsto \mu^\vee\sigma\varphi_{\alpha_0}\mu\varphi_{\alpha_0}^{-1}=\mu^\vee\sigma(\mu^\dagger)^\vee.\]
This concludes the proof of the theorem.
 \end{proof}

 The previous result could also be restated in the following equivalent way, which has the benefit of reducing the problem to working with automorphisms of $A$ as opposed to the dual variety:

	\begin{theorem}
    	If $K\leq A$ is a subgroup of type $D$ and $\alpha_0$ is a polarization with $K(\alpha_0)=K$, there is a bijection
    \[\mathcal{P}_D(K)\longleftrightarrow \mathrm{Aut}_+^s(A,K)/\mathrm{Aut}(A),\]
    where $\mathrm{Aut}_+^s(A,K)$ consists of all totally positive automorphisms of $A$ that are symmetric with respect to the Rosati involution of the polarization $\alpha_0$ on $A$ and send $K$ to itself, and $\mathrm{Aut}(A)$ acts on $\mathrm{Aut}_+^s(A,K)$ by $\mu:\sigma\mapsto\mu^{-1} \sigma \mu$. 		
	\end{theorem}

      \begin{proof}
    Using the previous theorem, we note that there is a bijection $\mathrm{Aut}_+^s(A^\vee) \to \mathrm{Aut}_+^s(A,K)$, given by the following: for every $\sigma_{\beta}= \varphi_\beta\varphi_{\alpha_0}^{-1}$ we have 
    $$\sigma_{\beta}^\vee =f_{\beta} =  \varphi_{\alpha_0}^{-1} \varphi_\beta \in \mathrm{Aut}_+^s(A).
    $$
    Furthermore, it is clear that $f_{\beta}(K)=K$. 
    \end{proof}

Let $\mathcal{H}_D$ be the set of all finite subgroups $K$ of $A$ of type $D$ such that there exists a polarization $\alpha$ with $K(\alpha)=K$. Then $\mathrm{Aut}(A)$ acts on $\mathcal{H}_D$ in the obvious way. We obtain the following corollary of Theorem \ref{counttheorem}: 

\begin{corollary}\label{verygeneral}
There is a bijection between
\[\mathcal{P}_D\hspace{0.5cm}\longleftrightarrow\bigsqcup_{[K]\in\mathcal{H}_D/\mathrm{Aut}(A)}\mathrm{Aut}_+^s(A^\vee)/G_K\]
where the action of $G_K$ on $\mathrm{Aut}_+^s(A^\vee)$ is as in Theorem \ref{counttheorem}.
\end{corollary}

We also get the following corollary:

\begin{corollary}
For any $D$ and any closed irreducible subvariety $Z\subseteq\mathcal{A}_g$, the function $p_D:Z\to\mathbb{N}\cup\{0\}$ is constant on a very general element of $Z$.
\end{corollary}
\begin{proof}
Via the rational representation and by looking at period matrices, it is clear that for a very general element of $Z$, the (rational representation of the) endomorphism algebra as well as the Rosati involution on the dual abelian variety are constant. This implies that $\mathrm{Aut}_+^s(A^\vee)$ is constant for a very general element of $Z$, and the action of $G_K$ as well.
\end{proof}

\begin{rem}
We note that when speaking of a very general element of a subvariety $Z\subseteq\mathcal{A}_g$, we mean the locus of principally polarized abelian varieties in $Z$ that have the same endomorphism ring of the lowest possible rank. We recall that the rank of the endomorphism algebra is an upper semi-continuous function on $\mathcal{A}_g$.
\end{rem}

Another striking and unexpected corollary of Theorem \ref{counttheorem} is the following:

\begin{corollary}\label{ppmin}
If $A$ has a polarization of type $D$, then $\pi(A)\leq\pi_D(A)$.
\end{corollary}
\begin{proof}
This happens because if $\alpha$ is a polarization of type $D$ with kernel $K$, then $G_K\leq\mathrm{Aut}(A)$, and so the number of elements in $\mathcal{P}_D(K)$ is greater than or equal to the number of elements in $\mathrm{Aut}_+^s(A^\vee)/\mathrm{Aut}(A)$.
\end{proof}

\section{Irreducible components of $\mathrm{Sing}(\mathcal{A}_g)$}\label{moduli1}

Let $p$ be an odd prime; then the cyclic group $C_p := \langle \rho \rangle$ of order $p$ has two rational irreducible representations: the trivial representation $\chi_0$, and another $W$ of degree $p-1$ where $W \otimes\mathbb{C} = U_1 \oplus \ldots\oplus U_{p-1}$, and $U_j$ are the non-trivial complex irreducible representations of $C_p$ given by $U_j(\rho) = w_p^j$, with $w_p$ a primitive $p$-th root of $1$. 
	
In particular, if $C_p$ acts on a principally polarized abelian variety $(A,\Theta)$ of dimension $g$, then its rational representation has the form $\rho_r(C_p) = 2b \, \chi_0 \oplus c\, W$, with $b$ and $c$ non-negative integers; therefore
	$$
	g = b + \frac{c (p-1)}{2}.
	$$

We have a decomposition
\[T_0(A)=V_0\oplus V_1\oplus\cdots\oplus V_{p-1},\]
where $V_i$ is the subspace where the automorphism acts as $w_p^i$. Define
\[n_i:=\dim_\mathbb{C}V_i.\]
We note that $b = n_0$ and $c = n_i +n_{p-i}$ for any $i >0$. 
    
	If $c>0$, then using the isotypical decomposition (cf. \cite[Section 1]{LR}), we get two abelian subvarieties $X,Y\subseteq A$, and the addition map 
	\[a:X\times Y\to A\]
	is a $\rho$-equivariant isogeny where $\rho$ acts on $X$ trivially and on $Y$ with finitely many fixed points. Moreover, $\dim X =b$ and $\dim Y= g-b = \frac{c (p-1)}{2}$; also,
	$a^*\Theta=\Theta_X\boxtimes\Theta_Y$ where $\Theta_X:=\Theta|_X$ and $\Theta_Y:=\Theta|_Y$. Note that $\ker a\simeq X\cap Y$ and $\ker \Theta|_X \simeq \ker \Theta|_Y \simeq X \cap Y$; these groups are all subgroups of the group of $p$-torsion points, and their order is an even power of $p$.  
	
	Observe that $Y$ is the image of $A$ under the endomorphism $1_A - \rho$ and $X$ is the connected component of its kernel containing $0$. 
	
	\begin{lemma}
		The number of fixed points of $\rho$ restricted to $Y$ is  $p^{c}$.
	\end{lemma}
	\begin{proof}
		By \cite[Prop. 13.2.5]{BL}, the number of fixed points of $\rho|_Y$ is equal to 
		\[\#\mathrm{Fix}(\rho|_Y)= p^{\frac{2\dim Y}{p-1}} = p^{c}. \]
		\end{proof}
		
	\begin{corollary}
		If $c=1$, then the isogeny $a:X\times Y\to A$ is an isomorphism of polarized varieties.
	\end{corollary}	
	
	\begin{proof}
		If $P \in X \cap Y$, then $P$ is a fixed point of $\rho$ in $Y$, but in this case $\#\mathrm{Fix}(\rho|_Y) = p$ is not a square, and therefore $P=0$ whence $a$ has trivial kernel.
	\end{proof}

We will continue with the same notation as above: let $p$ be an odd prime, let $n\geq 3$ be a fixed integer with $p\nmid n$, and let $\rho_n\in\mathrm{GL}(A[n])$ be an automorphism. 

\begin{definition}
We define the set $\mathcal{A}_{g}(p,\rho_n)$ to be the locus of principally polarized abelian varieties $(A,\Theta)$ that possess an automorphism $\sigma\in\mathrm{Aut}(A,\Theta)$ of order $p$ such that $\sigma|_{A[n]}$ is conjugate to $\rho_n$. 
\end{definition}

By \cite[Theorem 1.5]{GMZ}, $\mathcal{A}_g(p,\rho_n)$ is a closed irreducible subvariety of $\mathcal{A}_g$ contained in the singular locus, and each irreducible component of $\mathrm{Sing}(\mathcal{A}_g)$ is of such a form. Moreover, 
\[\dim\mathcal{A}_g(p,\rho_n)=\frac{n_0(n_0+1)}{2}+\sum_{i=1}^{(p-1)/2}n_in_{p-i}.\]
Now by the previous analysis that shows that an element $(A,\Theta)\in\mathcal{A}_g(p,\rho_n)$ splits isogenously as a product $X\times Y$, and moreover following \cite[Section 3.5]{Auff}, we have that $\mathcal{A}_g(p,\rho_n)$ is contained in the image of a morphism
\[\mathcal{A}_{b}^D\times\mathcal{A}_{g-b}^{\tilde{D}}\to\mathcal{A}_g,\]
where $D=(d_1,\ldots,d_b)$ is the type of $\Theta|_X$, $\tilde{D}$ is the type of $\Theta|_Y$, and $\mathcal{A}_b^D$ consists of all triplets $(B,\mathcal{L}_B,h)$, where $(B,\mathcal{L}_B)$ is a polarized abelian variety of type $D$ and $h:K(\mathcal{L}_B)\to\left(\mathbb{Z}/d_1\mathbb{Z}\times\cdots\times\mathbb{Z}/d_b\mathbb{Z}\right)^2$ is a symplectic isomorphism with respect to certain given pairings (which we will not be needing in this article).
Now the automorphism $\sigma$ preserves $X$ (on which it is trivial) and $Y$, and therefore it induces an element $\psi_k\in\mathrm{GL}(Y[k])$ for any $k\in\mathbb{N}$.

Putting everything together, we obtain the following result:

\begin{proposition}\label{moduli}
Fix $k\geq3$ and let $\mathcal{A}_{g-b}^{\tilde{D}}(p,\psi_k)$ be the space of all polarized abelian varieties $(Z,\mathcal{L}_Z)$ of type $\tilde{D}$ and dimension $g-b$ with a symplectic isomorphism as above that have an automorphism $\eta$ of order $p$ such that the restriction of $\eta$ to the $k$-torsion points of the abelian variety is conjugate to $\psi_k$ and $\eta$ restricted to $K(\mathcal{L}_Z)$ is the identity. Then $\mathcal{A}_g(p,\rho_k)$ is the image of the restriction
\[\mathcal{A}_{b}^D\times\mathcal{A}_{g-b}^{\tilde{D}}(p,\psi_k)\to\mathcal{A}_g.\]
 \end{proposition}

Inuitively speaking, we see that the abelian subvariety where the automorphism acts trivially can essentially ``move as much as it wants" in moduli.

 \section{Moduli spaces of abelian varieties with endomorphism structure}\label{moduli2}
 
 We recall a rather technical result that we will need from \cite[Chapter 9]{BL}. Let $\zeta_p$ be a primitive $p$th root of unity, and consider a representation $\rho:\mathbb{Q}(\zeta_p)\to M_g(\mathbb{C})$. Following \cite{BL}, a \textit{polarized abelian variety with endomorphism structure} $(\mathbb{Q}(\zeta_p),',\rho)$, where $'$ in this case is just complex conjugation on $\mathbb{Q}(\zeta_p)$, is by definition a triplet $(X,H,\iota)$, where $(X,H)$ is a polarized abelian variety and $\iota:\mathbb{Q}(\zeta_p)\hookrightarrow\mathrm{End}_\mathbb{Q}(X)\subseteq M_g(\mathbb{C})$ is such that $\iota$ and $\rho$ are equivalent representations, and the Rosati involution with respect to $H$ acts as complex conjugation on $\iota(\mathbb{Q}(\zeta_p))$.

 Now using the same notation as in the previous section, let $\mathcal{M}$ be a free $\mathbb{Z}$-submodule of $\mathbb{Q}(\zeta_p)^c$ of rank $(p-1)c$, and let $T\in M_c(\mathbb{Q}(\zeta_p))$ be a non-degenerate matrix. In \cite[Section 9.6]{BL}, an irreducible moduli space $\mathcal{A}(\mathcal{M},T)$ is constructed that parametrizes certain polarized abelian varieties with endomorphism structure $(\mathbb{Q}(\zeta_p),',\rho)$, and for any polarized abelian variety with endomorphism structure, there exists a submodule $\mathcal{M}$ and a non-degenerate matrix $T$ such that the polarized abelian variety with endomorphism structure belongs to $\mathcal{A}(\mathcal{M},T)$ (see \cite[Prop. 9.6.5]{BL}).

Going back to our context above, we see that our variety $\mathcal{A}_{g-b}^{\tilde{D}}(p,\psi_k)$ that parametrizes polarized abelian varieties $(Y,\mathcal{L}_Y)$ of type $\tilde{D}=(1,\ldots,1,d_1,\ldots,d_b)$ with an automorphism of order $p$ whose restriction to $Y[k]$ gives a representation equivalent to $\psi_k$ and whose restriction to $K(\mathcal{L}_Y)$ is trivial, along with a symplectic isomorphism $K(\mathcal{L}_Y)\to \left(\mathbb{Z}/d_1\mathbb{Z}\times\cdots\times\mathbb{Z}/d_b\mathbb{Z}\right)^2$, consists of elements that have an endomorphism structure $(\mathbb{Q}(\zeta_p),',\rho)$, where $\rho$ is just the analytic representation. 

Now let $(Y,\mathcal{L}_Y,f)\in \mathcal{A}_{g-b}^{\tilde{D}}(p,\psi_k)$. Then $Y$ has an automorphism $\sigma$ of order $p$, and we get an embedding $\iota:\mathbb{Q}(\zeta_p)\hookrightarrow\mathrm{End}_\mathbb{Q}(Y)$. Let $\mathcal{M}$ and $T$ be the data such that $(Y,\mathcal{L}_Y,\iota)\in\mathcal{A}(\mathcal{M},T)$.

\begin{proposition}\label{importantprop}
Every element in $\mathcal{A}(\mathcal{M},T)$ has an automorphism of order $p$ that fixes the kernel of the polarization. \end{proposition} 

\begin{proof}
This follows from the technical construction of the space $\mathcal{A}(\mathcal{M},T)$. Indeed, by the construction found at \cite[Pg. 264--265]{BL}, the rational representations of each element of $\mathbb{Q}(\zeta_p)$ does not depend on the element of the family chosen. Moreover, the representation of the polarization on the lattice is also constant for all members of $\mathcal{A}(\mathcal{M},T)$. Therefore, if \textit{only one element} satisfies the proposition, \textit{every} element of the family must satisfy it. Therefore, since by construction one element does satisfy the requirements, we are all set. 
\end{proof}

The most important application of this result for our purposes is the following:

\begin{corollary} The natural map
\[\mathcal{A}_{g-b}^{\tilde{D}}(p,\psi_k)\times_{\mathcal{A}_{\tilde{D}}}\mathcal{A}(\mathcal{M},T)\to \mathcal{A}_{g-b}^{\tilde{D}}(p,\psi_k)\]
is surjective.
\end{corollary}
\begin{proof}
This is just the same as saying that every polarized abelian variety that appears in $\mathcal{A}_{g-b}^{\tilde{D}}(p,\psi_k)$ also appears in $\mathcal{A}(\mathcal{M},T)$, which follows from the previous proposition.
\end{proof}

\begin{corollary}
A very general element of $\mathcal{A}_{g-b}^{\tilde{D}}(p,\psi_k)$ has the same endomorphism algebra as a very general element of $\mathcal{A}(\mathcal{M},T)$.
\end{corollary}

Again we recall that to each component $\mathcal{A}_g(p,\rho_n)$ of the singular locus, we can associate the list $(n_0,n_1,\ldots,n_{p-1})$, where $n_i$ is the dimension of the eigenspace of $\rho(\sigma)$ associated to the eigenvalue $\zeta_p^i$. 

Now using the previous results as well as \cite[Theorem 9.9.1]{BL}, we obtain the following theorem:

\begin{theorem}
We have that the endomorphism ring of a very general element of $\mathcal{A}_{g-b}^{\tilde{D}}(p,\psi_k)$ is exactly $\mathbb{Z}[\zeta_p]$, \textbf{unless} the dimension of this space is $0$, or $n_i=1$ for all $i$.
\end{theorem}
\begin{proof}
Translating \cite[Theorem 9.9.1]{BL} appropriately, we have that the general element has endomorphism algebra (that is, the endomorphism ring tensored with $\mathbb{Q}$) equal to $\mathbb{Q}(\zeta_p)$, except when the dimension of $\mathcal{A}(\mathcal{M},T)$ is 0, or when the order $p$ automorphism has exactly $p-1$ different eigenvalues (that is, when its minimal polynomial is the $p$th cyclotomic polynomial). Now Proposition \ref{importantprop} implies that the endomorphism ring contains $\mathbb{Z}[\zeta_p]$, and it is clear then that except for the two cases stated previously, there must be equality. 
\end{proof}

\begin{corollary}
Assume that $\dim \mathcal{A}_g(p,\rho_n)>0$. Then for a very general element $(A,\Theta)\in\mathcal{A}_g(p,\rho_n)$, we have that
\[\mathrm{End}_\mathbb{Q}(A)=\left\{\begin{array}{ll}\mathbb{Q}(\zeta_p)&\text{if }n_0=0,n_i\neq1\text{ for some }i\\
\mathbb{Q}\times\mathbb{Q}(\zeta_p)&\text{if }n_0>0,\sum_{i=1}^{(p-1)/2}n_in_{p-i}>0,n_i\neq1\text{ for some }i\end{array}\right.\]
\end{corollary}

\begin{rem} In the cases that either $\sum_{i=1}^{(p-1)/2}n_in_{p-i}=0$ or $n_i=1$ for all $i$, it is not clear what the endomorphism ring is in general. By \cite[Exercises 9.10.3]{BL}, we have that if $\sum_{i=1}^{(p-1)/2}n_in_{p-i}=0$, then a general element of $\mathcal{A}_g(p,\rho_n)$ is isogenous to the product of an abelian variety of dimension $b$ with a self-product of an abelian variety $Y$ such that $\mathrm{End}_\mathbb{Q}(Y)$ contains $\mathbb{Q}(\zeta_p)$. In the case when $n_i=1$ for all $i$, by \cite[Exercise 9.10.4]{BL}, we have that the general element of $\mathcal{A}_g(p,\rho_n)$ is isogenous to the product of a $b$-dimensional abelian variety with an abelian variety $Y$ such that $\mathrm{End}_\mathbb{Q}(Y)$ contains a totally indefinite quaternion algebra that contains $\mathbb{Q}(\zeta_p)$.
 \end{rem}

 \begin{rem}
We note that the case $p=2$ from Theorem \ref{main} follows immediately from the previous analysis.
 \end{rem}

 \begin{example}
There is a $3$-dimensional family of curves $C$ of genus $g=4$ admitting an automorphism $\eta$ of order $p=3$ such that $C \to C/\langle \eta \rangle$ has signature $(1;3,3,3)$;  that is, $C/\langle \eta \rangle$ has genus $1$ and $\eta$ fixes three points in $C$.

There is a symplectic basis $\{ \lambda_1 , \ldots, \lambda_4, \mu_1 , \ldots, \mu_4 \}$ for $H_1(C,\mathbb{Z})$ with respect to which the rational representation $\rho_r$ of $\eta$ is given by
$$
\rho_r(\eta) = \left( 
\begin{array}{cccccccc}
	0 & 0 & 1 & 0 & 0 & 0 & 0 & 0 
	\\
	1 & 0 & 0 & 0 & 0 & 0 & 0 & 0 
	\\
	0 & 1 & 0 & 0 & 0 & 0 & 0 & 0 
	\\
	0 & 0 & 0 & 0 & 0 & 0 & 0 & -1 
	\\
	0 & 0 & 0 & 0 & 0 & 0 & 1 & 0 
	\\
	0 & 0 & 0 & 0 & 1 & 0 & 0 & 0 
	\\
	0 & 0 & 0 & 0 & 0 & 1 & 0 & 0 
	\\
	0 & 0 & 0 & 1 & 0 & 0 & 0 & -1 
\end{array}\right)  = 2\chi_0 \oplus 3W
$$ 	

The matrices of the form 	
$$	
Z = \left(\begin{array}{cccc}
	\mathit{a} & \mathit{b} & -\mathit{c}^{2}+\mathit{b} & \mathit{c} 
	\\
	\mathit{b} & w \mathit{c}^{2}+\mathit{a} & w^{2} \mathit{c}^{2}+\mathit{b} & w \mathit{c} 
	\\
	-\mathit{c}^{2}+\mathit{b} & w^{2} \mathit{c}^{2}+\mathit{b} & w \mathit{c}^{2}+\mathit{c}^{2}+\mathit{a} & w^{2} \mathit{c} 
	\\
	\mathit{c} & w \mathit{c} & w^{2} \mathit{c} & w  
\end{array}\right)
$$	
with $\textup{Im}(Z) > > 0$ and $w$ a primitive third root of unity are the Riemann matrices of principally polarized abelian varieties fixed under the $\rho_r(\eta)$-action.

A calculation then gives us that the period matrix for $Y^{\vee}$ with respect to the basis given by the period matrix above is  $(D_{\vee} \ \tau_{\vee})$
with $D_{\vee} = \textup{diag}(3,3,1)$ and 
$$
\tau_{\vee} =  \left(\begin{array}{ccc}
	3w  & (w-1)\mathit{a} &  \mathit{a} 
	\\
	(w-1)\mathit{a} & 3\mathit{b} +\frac{1}{3} (w-1) \mathit{a}^2  & \frac{\mathit{a}^{2}}{6}-\frac{3 \mathit{b}}{2} 
	\\
	\mathit{a} & \frac{\mathit{a}^{2}}{6}-\frac{3 \mathit{b}}{2} &\mathit{b}
\end{array}\right)
$$
with $w$ a primitive third root of unity. As we have proved before, for generic values of $a$ and $b$, it is not hard to calculate by hand that the endomorphism algebra is isomorphic to $\mathbb{Z}[w]$, and is in particular of rank 2 over $\mathbb{Z}$, and the symmetric endomorphisms are a subgroup of rank 1. However, if we set $a=0$ and look at a generic value for $b$, then we get that the endomorphism algebra is of rank $6$, and the symmetric endomorphisms are of rank $4$. In this case it is very difficult to calculate by hand what the totally positive symmetric automorphisms should be, and in particular how many principal polarizations each element should have.
\end{example}
 
\section{Counting principal polarizations}

 Now we arrive at our initial problem: We wish to calculate the number of principal polarizations of a very general member of $\mathcal{A}_g(p,\rho_n)$. In this section we will assume that $n_i\neq1$ for some $i$, and that $\sum_{i=1}^{(p-1)/2}n_in_{p-i}>0$.  Let $(A,\Theta)\in\mathcal{A}_g(p,\rho_n)$ be a very general member with an automorphism $\sigma$, and as before let $X$ and $Y$ be the abelian subvarieties of $A$ where $\sigma$ acts trivially and with a finite number of fixed points, respectively. Since we are assuming that $(A,\Theta)$ is very general, by Section \ref{moduli1} we can assume that $X$ is very general, and therefore assume that

 \begin{itemize}
\item $\mathrm{rk}(\mathrm{NS}(X))=1$
\item $\mathrm{Hom}(X,Y)=0$.
 \end{itemize}

 Now let $\Xi$ be any principal polarization on $A$. Since $X$ is of Picard number 1, we have that $\Xi|_X\equiv m\Theta|_X$ for some $m\in\mathbb{Q}$. Now, since $X$ and $Y$ are complementary abelian subvarieties with respect to $\Theta$, we get that $|\deg a| =|X\cap Y|=|K(\Theta|_X)|$. By our assumptions on $X$ and $Y$, we have these are also complementary with respect to $\Xi$, and so $|\deg a|=|K(\Xi|_X)|$, we must have $m=1$; that is, $\Xi|_X\equiv\Theta|_X$. 

On the other hand, we have the equality $K(\Theta|_Y)=X\cap Y$ which must also be true for $\Xi|_Y$ since $X$ and $Y$ are complementary with respect to $\Xi$. We have therefore proven the following theorem:

\begin{theorem}
Let $p$ and $\rho_n$ be such that $\dim\mathcal{A}_g(p,\rho_n)>0$, and let $(A,\Theta)\in\mathcal{A}_g(p,\rho_n)$ be very general (in the previous sense). Then if $Y$ denotes the largest abelian subvariety of $A$ on which the order $p$ automorphism of $A$ acts with finitely many fixed points and $\mathcal{P}$ is the set of principal polarizations on $A$ modulo the automorphism group of $A$, we have a bijection
\[\mathcal{P}\longleftrightarrow\mathrm{Aut}_+^s(Y^\vee)/G_{K(\Theta|_Y)}.\]
\end{theorem}

Now under our hypotheses, we have that $\mbox{End}(Y)=\mathbb{Z}[\zeta_p]$. In this case, we see that
$$\text{Aut}(Y)=\mathbb{Z}[\zeta_p]^\times.$$
Since $\zeta_p$ is not only an automorphism, but also preserves the polarization, the Rosati involution acts on $\mathbb{Z}[\zeta_p]$ by conjugation. Therefore the set (which in this case is a group) of symmetric automorphisms of $Y$ is exactly
$$\text{Aut}^s(Y)=\mathbb{Z}[\zeta_p+\zeta_{p}^{-1}]^\times=:U_p.$$
A totally positive automorphism therefore corresponds to an element in $U_p$ such that all the eigenvalues of its minimal polynomial over $\Q$ are strictly positive real numbers. Equivalently it is an element $\eta\in U_p$ such that any embedding $\Q(\eta)\hookrightarrow\C$ sends $\eta$ to a positive real number. Let $U_p^+$ be the group of totally positive symmetric automorphisms.

\begin{lemma}
Given our hyptheses, $G_{K(\Theta|_Y)}=\mathbb{Z}[\zeta_p]^\times$.
\end{lemma}
\begin{proof}
Indeed, since $\zeta_p\in\mathrm{Aut}(Y,\Theta|_Y)$, we have that $\zeta_p\in G_{K(\Theta|_Y)}$. Now since every element of $\mathbb{Z}[\zeta_p]^\times$ is a polynomial in $\zeta_p$ with integer coefficients, and each integer and each power of $\zeta_p$ sends $K(\Theta|_Y)$ to itself, every element of $\mathbb{Z}[\zeta_p]^\times$ preserves $K(\Theta|_Y)$.
\end{proof}

Let $U_p^2$ be the group of squares of elements in $U_p$. We have that $U_p^2$ is a finite index subgroup of $U_p^+$.

\begin{proposition}\label{case =}
There is a bijection between $\mathrm{Aut}_+^s(Y^\vee)/G_{K(\Theta|_Y)}$ and $U_p^+/U_p^2$.
\end{proposition}
\begin{proof}
By what has been established, there is a bijection between $\text{Aut}_+^s(Y)$ and $U_p^+$. Now if $f\in G_{K(\Theta|_Y)}=\mathbb{Z}[\zeta_p]^\times$ and $g\in U_p^+$, we have that
$$f^\dagger gf=\overline{f}fg=|f|^2g.$$
Now each element of $\mathbb{Z}[\zeta_p]^\times$ is of the form $\zeta_p^ju$ where $u\in\mathbb{Z}[\zeta_p+\zeta_p^{-1}]^\times$. Therefore there exists $j$ and $u\in\mathbb{Z}[\zeta_p+\zeta_p^{-1}]^\times$ such that
$$f=\zeta_p^ju,$$
and so $|f|^2=u^2\in U_p^2$.
\end{proof}

\begin{definition} We define $\mathfrak{u}(p):=|U_p^+/U_p^2|$; this is just the narrow class number of $\mathbb{Q}(\zeta_p)$ divided by its class number. 
\end{definition}

Summing everything up, we have proved the main theorem of our paper:

\begin{theorem}
Let $Z:=\mathcal{A}_g(p,\rho_n)$ be a component of $\mathrm{Sing}(\mathcal{A}_g)$ such that $n_i\neq1$ for some $i$, and such that $\sum_{i=1}^{(p-1)/2}n_in_{p-i}>0$. Then a very general element of $Z$ has exactly $\mathfrak{u}(p)$ principal polarizations.
\end{theorem}

We note that the technical conditions are placed to assure us that the very general $Y$ that appears in the decomposition has endomorphism ring $\mathbb{Z}[\zeta_p]$.

\begin{proposition}
If $Z$ is as before, $(A,\Theta)\in Z$ is a very general element and $\Theta'$ is another principal polarization on $A$, then $(A,\Theta')\in Z$.
\end{proposition}
\begin{proof}
Indeed, if $Y$ is the largest abelian subvariety of $A$ where the automorphism $\sigma$ acts with finitely many points as before, we have that $\mathrm{End}(Y)=\mathbb{Z}[\zeta_p]$ is abelian, and therefore $\sigma$ fixes any principal polarization on $A$. This implies that for any principal polarization $(A,\Theta')$, there exists an automorphism $\sigma\in\mathrm{Aut}(A,\Theta')$ such that $\sigma|_{A[n]}$ is conjugate to $\rho_n$, and therefore $(A,\Theta')\in\mathcal{A}_g(p,\rho_n)=Z$.
\end{proof}

\section{Examples}

We will end this article with several interesting examples.

\subsection{Examples in infinitely many genera}
Let $k$ be a subfield of $\mathbb{C}$ and let $f\in k[x]$ be of degree $n\geq 5$ and such that its Galois group over $k$ is either $S_n$ or $A_n$. Then the main result of \cite{Zarhin} says that the endomorphism algebra of the Jacobian of (the normalization of the projective curve associated to) $y^p=f(x)$ over $\mathbb{C}$ is $\mathbb{Z}[\zeta_p]$ where $\zeta_p$ is a primitive $p$th root of unity. 

In particular, \cite[Proposition 2.2]{GMZ} implies that the Jacobian described in the previous example lies in some irreducible component $\mathcal{A}_g(p,\rho_n)$ of the singular locus of $\mathcal{A}_g$ whose general element has endomorphism algebra $\mathbb{Z}[\zeta_p]$, for $g=(p-1)(n-1)/2$ if $p$ does not divide $n$ and $g=(p-1)(n-2)/2$ if it does. In this case the underlying abelian variety of each element of $\mathcal{A}_g(p,\rho_n)$ appears as the underlying abelian variety of exactly $\mathfrak{u}(p)$ points in $\mathcal{A}_g$.

As a matter of showing explicit numbers, consider $JC$, the Jacobian of the genus 620 curve 
\[C:y^{311}=x^5+2x+2.\]
By our previous analysis, we have that $JC$ has $\mathfrak{u}(311)=1024$ principal polarizations! Now $\mathbb{Z}/311\mathbb{Z}$ acts on $C$ with signature $(0;311^6)$, and so $JC$ moves in a $3$-dimensional family of Jacobians with the same group action, and the generic Jacobian in this family also has $1024$ principal polarizations. It is not at all clear if all of these polarizations are from Jacobians. By running the numbers, we get that 
\[1\leq i\leq 62\Rightarrow n_i=4\]
\[63\leq i\leq 124\Rightarrow n_i=3\]
\[125\leq i\leq 186\Rightarrow n_i=2\]
\[187\leq i\leq 248\Rightarrow n_i=1\]
\[249\leq i\leq 310\Rightarrow n_i=0\]
This implies that the component of $\mathrm{Sing}(\mathcal{A}_g)$ that $JC$ belongs to is of dimension
\[\sum_{i=1}^{155}n_in_{311-i}=3(124-62)+4(155-124)=310.\]

\subsection{Examples in dimension 2} In dimension 2, it is a somewhat classic problem to determine if a Jacobian is isomorphic to a product of two elliptic curves. See, for example, the work done by Hayashida and Nishi in \cite{HN}. More recently, Kani \cite{Kani} studied the locus of genus 2 curves in $\mathcal{M}_2$ whose Jacobian is isomorphic to the product of two elliptic curves, and showed that if $T(d)$ denotes the locus of curves whose Jacobian is isomorphic to a product of two elliptic curves that are connected by a cyclic isogeny of degree $d$, then $T(d)$, if non-empty, is a finite union of irreducible curves each of which is birational to a certain (quotient of a) modular curve. In \cite[Theorem 1]{Kani}, he states that if Gauss's Conjecture is true, then for $d\geq463$, $T(d)\neq\varnothing$. This is interesting in our context, since Lange \cite{Lange2} shows as a corollary to his main theorem that if $d\to\infty$, then the number of principal polarizations on a product of two elliptic curves connected by a (minimal) isogeny of degree $d$ goes to infinity. In particular, since a product of two elliptic curves has a non-trivial automorphism group (that fixes the product polarization), we get the following:

\begin{proposition}
For $g=2$, if Gauss's Conjecture is true, then given $N\in\mathbb{N}$, there exists a $2$-dimensional Jacobian $J$ such that $\pi(J)\geq N$. Moreover, $J$ can be taken such that not all of the principal polarizations it has make it into a Jacobian, and some of the polarizations have a non-trivial automorphism group.
\end{proposition}

This shows that, although in this article we studied \textit{generic} principally polarized abelian varieties with non-trivial automorphisms, if we look at specific cases then we can find very strange behavior.

\end{document}